\documentclass[12pt, reqno]{amsart}
\usepackage{amsmath, amsthm, amscd, amsfonts, amssymb}
\usepackage[bookmarksnumbered, colorlinks, plainpages, ]{hyperref}
\usepackage{here}
\usepackage{graphicx,float}
\usepackage{psfrag}
\usepackage{mathrsfs}
\usepackage[T1]{fontenc}
\usepackage{hyperref}

\usepackage[all]{xy}
\usepackage{pgfplots}

\usepackage{xspace}

\usepackage{xcolor}
\definecolor{darkgreen}{rgb}{0,0.51,0.11}
\usepackage{blkarray}
\newtheorem{theorem}{Theorem}[section]
\newtheorem{lemma}[theorem]{Lemma}
\newtheorem{proposition}[theorem]{Proposition}
\newtheorem{corollary}[theorem]{Corollary}
\theoremstyle{definition}
\newtheorem{definition}[theorem]{Definition}
\newtheorem{example}[theorem]{Example}

\theoremstyle{remark}
\newtheorem{remark}[theorem]{Remark}
\numberwithin{equation}{section}

\begin{document}
\title[Square entropy]{Square Entropy and Uniform n-to-1 Bernoulli Transformations}

\author[]{Pouya Mehdipour$^*$}

\address{$^1$Departamento de Matemática,
\newline \indent Universidade Federal de Viçosa}
\email{pouya@ufv.br
}

\author[]{Somayeh  Jangjooye Shaldehi}


\address{$^2$Department of Mathematics,
\newline \indent Faculty of Mathematical Sciences,
\newline \indent Alzahra University}
\email{s.jangjoo@alzahra.ac.ir
}



\begin{abstract}
In this paper, we define the so-called square entropy and prove that n-to-1 full zip shift maps are intrinsically ergodic. Furthermore, we show that square entropy characterizes uniform n-to-1 transformations of $(m,l)$-Bernoulli type that are extended Bernoulli transformations.
\end{abstract}

\maketitle


The concept of intrinsic ergodicity was first established by Parry \cite{11} and Weiss \cite{12,13} for transitive shifts of finite type and all their subshift factors (sofic shifts), and by Bowen \cite{0133} for shifts with the specification property. A natural question in the study of shift spaces is whether these properties imply intrinsic ergodicity. 

In \cite{9}, the authors introduced the concept of  $(m,l)$-Bernoulli transformations, which generalize the previously known two-sided Bernoulli transformations. These transformations are defined by measure conjugacy to an extended shift map known as a zip shift. Zip shift maps are local homeomorphisms, and in case of finite symbolic sets,  serve as a useful framework for studying the measure-theoretic properties of finite-to-1 maps.

Our studies shows that exists a class of extended Bernoulli $n$-to-1 transformations ( of $(m,l)$-Bernoulli type) that is characterized by their Kolmogorov-Sinai (KS) entropy, and it is possible to provide a version of the Ornstein isomorphism theorem for this class of examples. However, as illustrated in Figure \ref{fig:1}, the general construction  of an $(m,l)$-Bernoulli transformation, leads to examples of maps with identical KS entropy, but which are not measure-theoretically conjugate. This observation led us to the development of an extended entropy formula, which culminates in the introduction of square entropy, the central subject of this paper.

In this context, we demonstrate that square entropy can be a candidate to classify uniform extended Bernoulli maps. We further show that a class of full zip shift maps exhibits intrinsic ergodicity with respect to square entropy. Specifically, we prove that these maps possess a unique measure of maximal entropy (MME), which is the most natural measure on subshifts and serves as a primary tool for analyzing their statistical properties (Section \ref{sec:5}). Subsequently, we define a topological version of square entropy and, using a variational principle for $n$-to-1 zip shift maps, we connect these two dynamical notions.

In what follows, in Section \ref{sec:2} we present some preliminary results for invertible measure preserving maps. In Subsection \ref{subsec:2.1} the definition of classical topological entropy and variational principle is given. In Subsection \ref{subsec:2.2} the zip shift dynamics are defined \cite{9}. The third Section introduces the notion of square measure and topological entropy. In Subsection \ref{subsec:3.1} and Subsection \ref{subsec:3.2} the square measure entropy and the square topological entropy are  established, correspondingly. In Section \ref{sec:4} we study the variational principal and intrinsic ergodicity of n-to-1 zip shift maps and, finally, in Section \ref{sec:5} we present a classification theorem for uniform n-to-1 maps of $(m,l)$-Bernoulli type, using the square entropy.

\section{Some preliminary results}\label{sec:2}
Let $(X,\mathcal{B},\mu)$ be a Borel probability space (or a Lebesgue Space \cite{8,10}) and $f:X\to X$ be a measure preserving map. Such dynamics equipped with an invariant measure is denoted by $(f,\mu)$ and we call them, measure dynamics. Take a finite measurable partition $\eta=\{X_1,X_2,\cdots,X_r\}$ of $X$. Given two partitions $\eta_1,\eta_2$ of $X$, let $\eta_1\vee\eta_2$ denote their sum, in which it is defined as following.
$$\eta_1\vee \eta_2=\{X_i^{1}\cap X_j^{2}| X_i^{1}\in \eta_1,\,X_j^{2}\in \eta_2\}.$$
Let associate some information function $I:X\to \mathbb{R}$ to any partition $\eta$ of $X$, where $I_{\eta}(x)=-\ln\,\mu(\eta(x))$. The entropy of the
partition $\eta$ is defined to be the mean of its information function,
\begin{equation}\label{P-entropy}
H_{\mu}(\eta) = \int I_{\eta} d\mu =
-\sum_{X_i\in \eta}\,\mu(X_i) \ln\, \mu(X_i),
\end{equation}
then for $\eta_1,\eta_2$ two finite measurable partitions of $X$, one can show that the following is valid \cite{012}.
\begin{equation}\label{Eq:2-1}
H_{\mu}(\eta_1 \vee \eta_2)\leq H_{\mu}(\eta_1)+H_{\mu}(\eta_2).
\end{equation}
Let $\eta=\{X_1,X_2,\cdots,X_r\}$ be a finite measurable partition for $X$ and define
$$\eta^n:=\bigvee_{i=0}^{n-1} f^{-i}(\eta),\,\,\,\,\,\,\,\,\,\, n \geq 1,$$
then $$h_{\mu}(f,\eta)=\lim_{n\to\infty}\frac{1}{n}H_{\mu}(\eta^n),$$
is the entropy of $f$ with respect to partition $\eta.$ 
The following Lemmas are known and useful.
\begin{lemma}\label{3.1}\cite{012}
	Suppose that $\{a_m\}_{m=1}^{+\infty}$ is a sequence satisfying $\inf \frac{a_m}{m}>-\infty$ and for all $m,n$ one has $a_{m+n}\leq a_m+a_n$. Then $\lim_{m\to +\infty}\frac{a_m}{m}$ exists and equals the $\inf \frac{a_m}{m}.$
\end{lemma}

\begin{lemma}\label{3.2}\cite{012}
	Every finite measurable partition $\eta=\{X_1,X_2,\cdots,X_r\}$ has finite entropy: $H_{\mu}(\eta)\leq \frac{1}{\#(\eta)}$ (where $\#$ stands for cardinality)
	and the identity holds if and only if $\mu(X_i)=\frac{1}{\#(\eta)}$ for every $X_i\in \eta$.
\end{lemma}

The Kolmogrov-Sinai entropy (KS-entropy) of the measure dynamical system $(f,\mu)$ is defined as, 
\begin{equation}\label{B-ent}
h_{\mu}(f)=\sup_{\eta}h_{\mu}(f,\eta).
\end{equation}

The following is a version of Kolmogrov-Sinai Theorem, useful in calculating the entropy.
\begin{theorem}\label{Thm:S-K}\cite{012}
	Let $f$ be a continuous map with $\mathcal{P}_1<\mathcal{P}_2<\cdots < \mathcal{P}_n<\cdots$ being a non-decreasing sequence of partitions with finite entropy such that $\bigcup_{n=1}^{\infty}\mathcal{P}_n$ generates
	the $\sigma$-algebra of measurable sets, up to measure zero. Then,
	$$h_{\mu}(f)=\lim_{n\to +\infty}h_{\mu}(f,\mathcal{P}_n).$$
\end{theorem}

\textbf{Notation.}
	Let us, for some reason that appears in future sections, denote this entropy by \textbf{$h_{\mu}^-(f,\eta)$}.	We may call this entropy the \textbf{backward} or \textbf{$(KS)^-$ entropy }of $f$.


If $S$ is a finite set of alphabets, then by full shift we mean the
collection of all bi-infinite sequences of symbols from $ S$. The full $ l $-shift is the full shift over the alphabet set $S=\{0,1,2,...,l-1\}$. The shift map $ \sigma$ on the full shift space $ \Sigma_{S}=S^{\mathbb Z} $ maps a point $x$ to
the point $  y=\sigma(x)$ whose $i$th coordinate is $  y_{i}=x_{i+1}$. One can equip $\Sigma_S$ with a metric in order to obtain a metric space and define the cylinder sets of the form 
$$C_{i_1,\dots,i_k}^{s_{i_1},...,s_{i_k}}=\{(x_n)_{n\in\mathbb{Z}}\in\Sigma_{S}| x_{i_1}=s_{i_1},\dots, x_{i_k}=s_{i_k}\,s.t.\, x_{i_{j}}\in S, 1\leq j\leq k \},$$ 
which are the open sets of the metric topology. The collection of all such cylinder sets provide a basis for the metric topology and proceed to the Borel $\sigma$-algebra.

Let $p=(p_{0},p_{1},\dots,p_{l-1})$ be a probability vector (i.e. $p_{i}\geq 0$ and $\sum_{i=0}^{l-1}p_{i}=1$ for any $0\leq i\leq l-1$). We then define 
$\mu(C_{i_1,\dots,i_m}^{s_{i_1},\dots,s_{i_k}})=p_{s_{i_1}}\dots p_{s_{i_k}},$
and obtain a Borel probability space $(\Sigma_{S},\mathcal{B},\mu)$.

Bernoulli transformations provide a wide variety of ergodic and dynamic properties, which make them of great importance in the study of measure dynamical systems. Given a measure dynamics $(f,\nu)$, we say that it has \textit{Bernoulli property}, if $(f,\nu)$ is conjugated (mod-0) to a $(\sigma,\mu)$, with $\sigma: \Sigma_{S}\to \Sigma_{S},$ being a two-sided shift homeomorphism and $P=(p_0,\cdots, p_{l-1})$ a probability distribution associated to the symbolic set $S=\{0,1,...,l-1\}$. Once $p_i=\frac{1}{l}$ for all $0\leq i\leq l-1$, we call $\mu$ a \textbf{$\frac{1}{l}$-uniform measure} or simply a \textbf{uniform measure}. 
\begin{definition} [\textbf{Isomorphic Transformations}]\label{I-T}
	Let $(X_1,\mathcal{B}_1,\mu_1)$ and $(X_2,\mathcal{B}_2,\mu_2)$ be measure spaces and $T_1:X_1\rightarrow X_1,\,$ $ T_2:X_2\rightarrow X_2$, measure preserving transformations. Then $T_1$ is \textbf{isomorphic } (or conjugated mod-0) to $T_2$ if there exists $M_1\in\mathcal{B}_1,\,M_2\in\mathcal{B}_2$ with $\mu_1(M_1)=\mu_2(M_2)=1,$ such that, 
	\begin{description}
		\item[(i)] $T_1(M_1)\subseteq M_1,\, T_2(M_2)\subseteq M_2;$
		\item[(ii)] there exists an invertible measure preserving $\varphi :M_1\rightarrow M_2$ with $\varphi_{|_{M_{1}}} \circ T_{1}=T_{2}\circ \varphi_{|_{M_{2}}}$.
	\end{description}
\end{definition}

\begin{definition}[\textbf{Bernoulli Transformations}]
	The invertible measure preserving transformation $f:X\rightarrow X$ defined on a Lebesgue space $(X,\mathcal{B}, \mu)$ is called a \textit{Bernoulli transformation}, if it has the Bernoulli property.
\end{definition}

\begin{definition}[\textbf{One-sided Bernoulli Transformations}]
	The non-invertible measure preserving transformation $f:X\rightarrow X$ defined on a Lebesgue space $(X,\mathcal{B}, \mu)$ has \textit{one-sided Bernoulli property} if it is conjugated mod-0 with a one-sided shift map on a set $S$, of finite elements, where $P=(p_1,\cdots, p_n)$ is a probability distribution on elements of $S$
\end{definition}

Let $\sigma:\Sigma_S\to \Sigma_S$ be a two-sided (one-sided) shift in a finite set $S$ with $l$ alphabets.
One constructs a nested sequence of partitions on basic cylinder sets (i.e. $\mathcal{P}_1=\{C_i^{s_i}|s_i\in S\}$) that their finite intersections generate the Borel $\sigma$-algebra.
Let $\mathcal{P}_n=\{C_{0,\cdots,n-1}^{s_{i_1},\cdots,s_{i_{n-1}}}\}_{s_{i_k}\in S},\,n\geq 1$. 
Then using Kolmogrov-Sinai Theorem \cite{012} 
the entropy of the two-sided (one-sided) shift as an example of a two-sided (one-sided) Bernoulli map with a uniformly distributed probability $p_i$ on elements of $S$, is as follows.
$$h_{\mu}(\sigma)=\lim_{n\to\infty}\frac{1}{n}H_{P}(\mathcal{P}^n)=-\sum_{s_i}p_{s_i}\ln p_{s_i}.$$
The Kolmogrov-Sinai entropy is an invariant of the conjugacy (mod-0), indeed any Bernoulli map $(f,\nu)$ conjugated with a uniform two-sided (one-sided) shift on $S$ alphabets has measure entropy equal $h_{\nu}(f)=h_{\nu}(\sigma)$.

\bigskip
\subsection{Topological entropy and Variational principle}\label{subsec:2.1}
Let $(X, d)$ be a compact metric space and $f: X\to X$ a continuous map. For $\epsilon > 0$ we will say that a set $E \subseteq X$ is $(n, \epsilon)$-spanning for $f$ if, for every $x \in X$, there is a $y \in E$ such that $d(f^j(x),f^j( y)) < \epsilon$ for $0\leq  j < n$. Note that by continuity of $f$ and compactness of $X$,  there are always finite $(n, \epsilon)$-spanning sets for every $\epsilon > 0$ and $n \geq 1$. Let $r_n(f, \epsilon)$ denote the size of the $(n, \epsilon)$-spanning set with fewest number of elements. Then
$$ r(f, \epsilon) := \limsup_{n \to \infty} \frac{-\ln r_n(f, \epsilon)}{n}, $$
the growth rate of $r_n(f, \epsilon)$ as $n \to \infty$. As $r(f, \epsilon)$ is non-decreasing as $\epsilon$ goes to $0$, so we may define
$$ h(\varphi) := \lim_{\epsilon \to 0} r(f, \epsilon). $$
This quantity is called the topological entropy of $f$.

Now let $f = \sigma$ be a bilateral shift map defined on $X \subseteq \Sigma_S$, equipped with the usual metric $d(x, y) = \frac{1}{2^{M(x, y)}}$, where $M(x, y) = \min\{|i| : x_i \neq y_i\}$. Then, one can verify that $r_n(f, \frac{1}{2^k}) = |B_{n+2k}(X)|$. Let $N = n + 2k$, then the following entropy formula for $\sigma$ is derived:
\begin{equation}\label{eq:T-ent}
h_{\text{top}}(\sigma) = \lim_{N \to \infty} \frac{1}{N} \ln |B_N(X)|. 
\end{equation}
Here, $B_N(X)$ represents the collection of permissible words for $N \geq 1$ and $B_N\in S^N$.
\begin{definition}\label{Def:gen}
	Let $f:X\to X$ be an n-to-1 local homeomorphism. Then we call a finite cover $\alpha,$ a \textit{Generator} for $f,$ if for any $\epsilon>0$ there exists $N>0$ such that the cover $\vee_{i=-N}^{N}f^{-i}(\alpha)=\{A_1,\dots,A_k\}$ consists of open sets each of which with diameter at most $\epsilon$. i.e. $\sup_i\,\,diam(A_i)<\epsilon$.
\end{definition}
When $f$ is a homeomorphism, we have the following known Proposition.
\begin{proposition}\label{Prop:0}\cite{8}
	If $f:X\to X$ is a homeomorphism and $\alpha$ is a generator for $f$, then $h_{top}(f)=h_{top}(f,\alpha).$
\end{proposition}

\begin{example}
	Let $\sigma$ be a full bilateral shift map defined on $\Sigma_S$ where $S=\{0,1,\dots, l-1\}$.  Then by \eqref{eq:T-ent}, the topological entropy $h_{top}(\sigma)=\ln l$
\end{example}

\begin{definition}[\textbf{Expansivity}]
	Let $(X,d)$ be a compact metric space and $f:X\to X$ a continuous dynamical system. We say that $f$ is an expansive map if there exists some $\delta>0$ such that for any $x,y\in X$, exists some $n\in\mathbb{Z}$ such that $d(f^n(x),f^n(y))>\delta$.
\end{definition}

\begin{theorem}\label{thm:vp}\cite{012}
	For $f:X\to X$ a continuous map defined on the compact metric space $X$. Then,
	$$h_{top} (f ) = \sup\{h_{\mu}(f) : \mu \in \mathcal{M}(X,f)\}.$$
	Where $\mathcal{M}(X,f)$ represents the set of all invariant probability measures.
\end{theorem}
The $f$-invariant measure $\mu$ in which $h_{top} (f)=h_{\mu}(f),$ is called the \textbf{"measure of maximal entropy"}. Whenever the measure of maximal entropy is unique, then $X$ is called \textbf{"intrinsically ergodic"}.

Let $S$ and $Z$ denote two sets of finite alphabets. Set $\sigma^{L}:\Sigma_{S}\to\Sigma_{S}$ as the \textbf{left} shift map defined on $\Sigma_S=\prod_{-\infty}^{+\infty} S$ and   $\sigma^{R}:\Sigma_{Z}\to\Sigma_{Z}$ as the \textbf{right} shift map defined on $\Sigma_Z=\prod_{-\infty}^{+\infty} Z$.

The following lemma is an adaptation from \cite{8}.
\begin{lemma}\label{Lem:m-e}
	The two-sided shift map $ \sigma^i, i=L,R$ has unique measure of maximal entropy and this unique measure is the $\frac{1}{l}$-uniform measure.
\end{lemma}
\begin{proof}
	We know $h_{top}(\sigma)=\ln l$. Suppose that for some $\sigma$-invariant $\mu$, $h_{\mu}(\sigma)=\ln l$. Let $\eta=\{C_{0}^j|j\in S\}$ (i.e. $C_{0}^j=\{(x_n)_{n\geq 0}: x_{0}=j\}$) be a generator of the $\sigma$-algebra. Then by \eqref{Eq:2-1} and Lemmas \ref{3.1} and \ref{3.2}, $\ln l=h_{\mu}(\sigma)\leq \frac{1}{n}H_{\mu}(\vee _{i=0}^{n-1}f^{-i}\eta)\leq\frac{1}{n}\ln l^{n}$. 
	Therefore each member of $\vee _{i=0}^{n-1}f^{-i}\eta$ has measure $(\frac{1}{l})^n$ and hence $ \mu $ is the $\frac{1}{l}$-uniform measure.
\end{proof}

From Theorem \ref{thm:vp} the following corollary arises.
\begin{corollary}\label{cor:1}
	The variational principle is valid for two-sided shift maps $\sigma^R$ and $\sigma^L$.
\end{corollary}

\begin{definition}[\textbf{An m-to-1 local homeomorphism}]\label{m-t-1 lh}
	Let $X$ be a compact metric space, and $\mathcal{P}$ be a finite measurable partition of $X$ (where $X$ is a disjoint union of the elements of $\mathcal{P}$ (mod 0) with $\#(\mathcal{P})=m$. Denote the elements of the partition by $X_i, i=1,\cdots,m$. Then we say that $f:X\to X$ is a m-to-1 local homeeomorphism, whenever $f_{{|}_{X_i}}:=f_i: X_i\to f(X)$  is a homeomorphism for any $i=1,\cdots, m$, $X_i\in \mathcal{P}$. 
\end{definition} 

\subsection{Zip shift space}\label{subsec:2.2}

In this subsection, we describe the zip shift maps \cite{7}. 
Let $Z=\{a_1,a_2,\cdots, a_m\}$ and $S=\{0,1,\dots, l-1\},$ be two collections of symbols that $m\leq l$ and  $\tau:S\to Z$ a surjective map. 
Consider $\prod_{-\infty}^{+\infty}S$ and let $\bar t=(t_i)\in\prod_{-\infty}^{+\infty}S$. Then, to any such point $\bar t=(t_i)_{i\in \mathbb{Z}}$ correspond a point $\bar x=(x_i)_{i\in \mathbb{Z}}$, such that
\begin{equation}\label{Z}
x_i=\left\{\begin{tabular}{ll}
$t_i\in S \hspace{7mm} $\,\,\,\,\quad\quad\quad$\forall i\geq 0$\\
$\tau(t_{i})\in Z \hspace{7mm} $\,\,\,\quad\quad $\forall i<0$. 
\end{tabular}\right.
\end{equation}
Define 
$\Sigma= \{\bar x=(x_i)_{i\in \mathbb{Z}}\ : x_i \textrm{\ satisfies} \,(\ref{Z})\}.$\\ We equip $\Sigma$ with a distance $\bar{d}$. 
Let $M:\Sigma\rightarrow\mathbb{N}\cup \{0\}$ be given as, $$M(x,y)=\left\lbrace\begin{array}{lr}
\infty,\quad\quad\quad\quad\quad\quad\quad\quad\quad\text{ if } x=y,\\
\min\{|i|;\,x_i\neq y_i \},\quad\quad\quad\text{ if }x\neq y,
\end{array}\right.$$
then, $\bar{d}(x,y)=\frac{1}{2^{M(x,y)}}$ is a metric which induces a topology on $\Sigma.$ 
When $\Lambda_1$ and $\Lambda_2$ are two subsets of $\Sigma$, let define 
$d(\Lambda_1,\,\Lambda_2)=\min\{\bar d(\omega_1,\,\omega_2):\, \omega_1\in \Lambda_1,\,\omega_2\in \Lambda_2\}$.
Consider the metric space $(\Sigma, \bar{d})$.
We define $\sigma_{\tau}:\Sigma\to\Sigma$ as follows.
\begin{eqnarray}\label{ZS}
(\sigma_{\tau}(\bar x))_i=\left\{\begin{tabular}{ll}
$x_{i+1} \,\,\,\quad\quad \text{if}\,\,i\neq -1,$ \\
$\tau(x_{0}) \quad\quad \text{if}\,\,i=-1.$
\end{tabular}\right.
\end{eqnarray}
We call this map \textbf{Zip shift} and the pair $(\Sigma, \sigma_{\tau})$ is called the\textbf{ (full) Zip shift Space} on $(m,l)$ symbols. Unless otherwise specified we denote a full zip shift space only by $\Sigma$.
Now consider $\Sigma$ as a full shift space and let $B_n(\Sigma)$ be the set of all words of length $n$ in $\Sigma$. Then the set $\mathcal{L}=\bigcup_{n\geq 0}B_n(\Sigma)$
is the set of all \textit{admissible} words or the \textit{language} of $\Sigma$. Any $X\subseteq \Sigma$ which is $\sigma_{\tau}$-invariant and closed is called a "\textit{sub-zip shift space}". Examples of sub-zip shift spaces are zip shifts of finite type and $M$-step zip shifts which include a set of forbidden words from $B_n(\Sigma)$ for some $n>0$ and are defied extensively in \cite{7}.
\begin{example}\label{ex:3.1}
	Let $S=\{0,1,2,3\}$ and $Z=\{a,b\}$. Then, the corresponded onto map $\tau:S\to Z$ is defined as $\tau(0)=\tau(2)=a$ and $\tau(1)=\tau(3)=b$.  
	Let $\bar x=(x_n)_{n\in\mathbb{Z}}=(\cdots a\,b\,a\,b\,b\,\textbf{.}\,1\,0\,3\,1\,1\,2\,\cdots).$  One can verify that $$\sigma_{\tau}((\cdots\,a\,b\,a\,b\,b\,\textbf{.}\,1\,0\,3\,1\,1\,2\,\cdots))=(\cdots\,a\,b\,a\,b\,b\,b\,\textbf{.}\,0\,3\,1\,1\,2\,\cdots),$$
	and 
	$$\sigma_{\tau}^2((\cdots\,a\,b\,a\,b\,b\,\textbf{.}\,1\,0\,3\,1\,1\,2\,\cdots))=(\cdots\,a\,b\,a\,b\,b\,b\,a\,\textbf{.}\,3\,1\,1\,2\,\cdots).\blacksquare$$ 
	In \cite{9} the authors show that this zip shift map is conjugated (mod-0) with a 2-to-1 baker's transformation represented in Figure \ref{fig:1}.		
\end{example}

\begin{figure}[h]
	\begin{center}
		\includegraphics[width=0.65\textwidth]{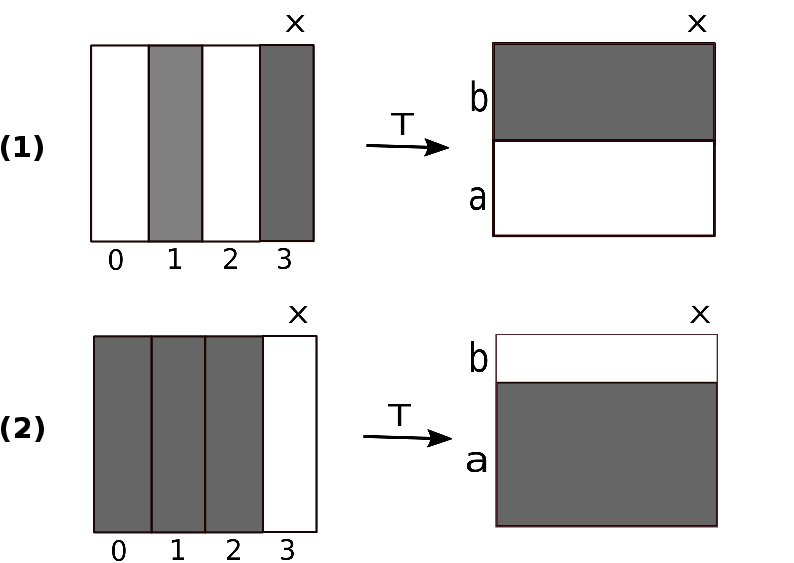}
		\caption{A 2-to-1 and a finite-to-1 Baker's transformation. Both are $(2,4)$-Bernoulli maps with the same Kolmogorov-Sinai (KS) entropy.}
		\label{fig:1}  
	\end{center}
\end{figure}

\begin{example}\label{Ex:2}
	Let $S=\{0,1,2,3\},\, Z=\{a,b\}$, with $F=\{ab,ba\}$ being the set of the forbidden words. Then consider the $X_{F}\subseteq \Sigma$ as a sub-zip shift space where the associated dynamics is $\sigma_{\tau}:X_F\to X_F$ with $\tau:S\to Z$ being $\tau(0)=\tau(1)=a$ and $\tau(2)=\tau(3)=b$. 
	
	
\end{example} 
\begin{proposition}\cite{9}
	The followings are valid for $\sigma_\tau:\Sigma \to \Sigma$.
	\begin{itemize}
		\item $\sigma_{\tau}(\Sigma)=\Sigma.$
		\item $\sigma_{\tau}$ is a local homeomorphism.
	\end{itemize}
\end{proposition}

\begin{proposition}
	The zip shift maps are expansive local homeomorphisms.
\end{proposition}

\begin{proof}
	Let $\sigma_{\tau}$ represent a zip shift map on $(m,l)$ symbols and take $\delta=1/2$. For any $x,y \in \Sigma$ recall that  $\bar d(x,y)=\frac{1}{2^{M(x,y)}}$ where 
	$$M(x,y)=\left\lbrace\begin{array}{lr}
	\infty,\quad\quad\quad\quad\quad\quad\quad\quad\quad\text{ if } x=y\\
	\min\{|i|;\,x_i\neq y_i \},\quad\quad\quad\text{ if }x\neq y
	\end{array}\right.$$
	Let $x\neq y $. Then there exists some $i\in \mathbb{Z}$ such that $x_{i}\neq y_{i}$. Let $i$ be the least in modulus of such $i$. If $i>0$, $d(\sigma_{\tau}^{i}x,\sigma_{\tau}^{i}y)=\frac{1}{2^0}>1/2$  and if $i<0$, then $d(\{\sigma_{\tau}^{i}x\},\{\sigma_{\tau}^{i}y\})=\frac{1}{2^0}>1/2$. Indeed the zip shift map is an expansive dynamical system. 
\end{proof}

Let $S=Z=\{0,1,\cdots,l\}$, then the known \textit{two-sided shift homeomorphism}, $\sigma:\Sigma_{S}\to \Sigma_{S}$ where $\Sigma_{S}=\prod_{-\infty}^{\infty}S$, is a zip-shift map on $l$ symbols, in which the onto map $\tau:S\to S$ is the identity. 
The basic cylinders on a zip shift space are defined as follows.
\begin{equation}\label{B-C}
C_i^{s_i}=\{x\in\Sigma\, |\,x_i=s_i\,\},
\end{equation}
such that if $i<0,\,s_i\in Z,$ and if $\,i\geq 0,\,s_i\in S$.
The $C_i^{s_i}$ presents the set of all sequences, that have $s_i$ in the $i-$th entry.  For $i,n\in \mathbb{Z}$ and $\ell\in \mathbb{N}\cup \{0\}$, one can define a general cylinder set as follows. 
$$C_{i_1\dots,i_k}^{s_{1},\dots,s_{k}}=\{ (t_n)\in \Sigma|t_{i_1}=s_1,\cdots,t_{i_k}=s_k; s_1,\dots,s_k\in S\cup Z\},$$
where  $s_j\in Z,\, \text{if,}\,i_j<0,\,\text{and}\,s_j\in S,\, \text{if,}\,i_j\geq 0$ ($1\leq j\leq k$). As cylinder sets are defined independently, one observes that, $$C_{i}^{s_i,\cdots,s_{i+l}}=C_{i}^{s_i}\cap C_{i+1}^{s_{i+1}}\cdots \cap C_{i+l}^{s_{i+l}},$$ and any general cylinder set can be produced in this way.  Such cylinder sets are clopen subsets in the metric topology of $(\Sigma, \bar d)$.  The set of all such cylinder sets, form a basis for the topology and the metric space $(\Sigma,\bar d)$ is compact, totally disconnected and perfect. Indeed it is a Cantor set. One can consider the Borel $\sigma$-algebra (denoted by $\mathcal{C}$) generated by the cylinder sets and transform the zip shift space $(\Sigma,\sigma_{\tau})$ into a measurable space $(\Sigma, \mathcal{C})$.

Consider the zip shift space $(\Sigma, \sigma_{\tau})$ defined on two set of symbols. The sets $Z=\{a_1,a_2,\cdots, a_m\}$ and $S=\{0,1,\dots, l-1\},$ equipped with onto map $\tau:S\rightarrow Z$. Let, $P_S=(p_0, \cdots,p_{l-1})$ be a probability distribution on $S$. Using $\tau$ one can induce a probability measure  $P_{Z}=(p'_{a_1},\cdots,p'_{a_m})$ on $Z$. To see that, take $p'_{a_i}=\sum_{s_j\in \tau^{-1}(a_i)}(p_{s_j})$, for any $1\leq i\leq m$. 
As basic cylinder sets are independent by definition, one sets $\mu(C_{i}^{s_j,\cdots,s_{j+k}}):=p_{s_j}.\,\dots  \,.p_{s_{j+k}},$ with
\begin{equation}\label{Eq:1}
\mu(C_i^{s_i}):=\left\lbrace\begin{array}{ll}
p_{s_i} & \quad\quad\text{if } s_i\in S; \\
p'_{s_i} & \quad\quad\text{if } s_i\in Z,
\end{array}\right. 
\end{equation}
and this provides a probability measure space $(\Sigma, \mathcal{C}, \mu)$. Once $p_i=\frac{1}{l}$ for all $0\leq i\leq l-1$, we call $\mu$ a $\frac{1}{l}$-uniform measure or simply a \textbf{uniform measure}.
\begin{proposition}\label{prop:3.1}
	Consider the probability space $(\Sigma,\mathcal{C},\mu)$ with $P_S$ as the probability distribution on $S$. Then $\sigma_{\tau}$ preserves the measure $\mu$, if and only if  $P_Z$ be
	such that, 
	\begin{equation}\label{Eq:inv-m}
	p'_{a_i}=\sum_{s_j\in \tau^{-1}(a_i)}(p_{s_j}),\quad\quad 1\leq i\leq m.
	\end{equation}
\end{proposition}

\begin{proof}
	The family of all cylinder sets of the form $C_{i}^{s_i,\cdots,s_{i+k}}$ generates the Borel $\sigma$-algebra  $\mathcal{C}$. First let assume that \eqref{Eq:inv-m} is satisfied. We show that for any  $C_{i}^{s_i,\cdots,s_{i+k}}\in \mathcal{C}$,
	$(\sigma_{\tau})_*(C_{i}^{s_i,\cdots,s_{i+k}})=\mu(C_{i}^{s_i,\cdots,s_{i+k}}),$
	where $(\sigma_{\tau})_*(\mu)=\mu(\sigma_{\tau}^{-1}).$ It may happen a number of cases:
	\begin{itemize}
		\item[I)] $C_{i}^{s_i,\cdots,s_{i+k}}$ is a cylinder set with $i\geq 0$ ($p_{s_i}$ 's are induced from $P_S$) $\underline or$ $i+k<-1,$ ($p'_{s_i}$s are induced from $P_Z$) then: $$\mu(\sigma_{\tau}^{-1}(C_{i}^{s_i,\cdots,s_{i+k}}))=\mu(C_{i+1}^{s_i,\cdots,s_{i+k}})=p_{s_i}.\,\dots  \,.p_{s_{i+k}}.$$
		\item[II)] $C_{i}^{s_i,\cdots,s_{i+k}}$ is a basic cylinder with $i\leq-1$ and $i+k\geq 0$:\\ Note that when $s_{i}\in Z$, $\mu(\sigma_{\tau}^{-1}(C_{-1}^{s_{i}}))=\mu(C_{-1}^{s_{i}})$. Because if $q=\#(\tau^{-1}(s_{i}))$, then by definition, for all $1\leq i\leq q$, 
		\begin{equation}\label{Eq:2.1}
		\sigma_{\tau}^{-1}(C_{-1}^{s_{i}})= \bigcup_{s_j\in \tau^{-1}(s_{i})} C_{0}^{s_j},	
		\end{equation} and 
		\begin{equation}\label{Eq:2}
		\,p'_{s_{i}}=\mu(C_{-1}^{s_{i}})=\sum_{s_j\in \tau^{-1}({s_i})}\mu(C_{0}^{s_j})=\sum_{s_j\in \tau^{-1}(s_{i})} p_{s_j}.	
		\end{equation} 
		Observe that $s_{j}\in S$ and $p_{s_j}$s are induced from $P_S$. Consequently, $$\mu(\sigma_{\tau}^{-1}(C_{i}^{s_i,\cdots,s_{i+k}}))=\mu((C_{i+1}^{s_i,\cdots,s_{i+k}}))=p_{s_i}.\,\dots  \,.p_{s_{i+k}}.$$
		
		Hence, $\sigma_{\tau}$ is a measure preserving map.
	\end{itemize}	
	For the converse part, note that if there exists some $i$ such that $p'_{a_i}\neq\sum_{s_j\in \tau^{-1}(a_i)}(p_{s_j})$, then \eqref{Eq:2} will be violated and consequently the measure can not be an invariant measure. Indeed by contradiction, it becomes correct.	 
\end{proof}

\begin{remark}\label{rem:2}
	The  above  proposition indicates that the invariant measures of $\Sigma$ are of the form \eqref{Eq:1} satisfying \eqref{Eq:inv-m}.
\end{remark}

\subsection{Entropy for (full) Zip shift maps.}
Let $\sigma_{\tau}:\Sigma\to \Sigma$ denote the two-sided zip shift map defined on symbolic sets $S=\{0,1,\dots, l-1\}$ and $Z=\{a_1,\dots, a_m\}$ with transition map $\tau: S\to Z$. One can use Theorem \ref{Thm:S-K} and \eqref{eq:T-ent} in order to calculate the measure and topological entropy of a zip shift map. Using \eqref{eq:T-ent} the topological entropy $h_{top}(\sigma_{\tau})= \ln l$. Howbeit, using Lemma \ref{3.2} and Theorem \ref{Thm:S-K}, as for $\mathcal{P}_n= \bigvee_{-n}^{n} \mathcal{P}, n\in\mathbb{Z},$ with $\mathcal{P}=\{C_{0}^{0},\dots,C_{0}^{l-1} \}$, the set $\bigcup_{0}^{n} \mathcal{P},$ generates the Borel $\sigma$- algebra, the measure entropy $h_{\mu}(\sigma_{\tau})=\lim_{n\to \infty} h_{\mu}(f,\mathcal{P}_n)$. Note that, $h_{\mu}(\mathcal{P}_n)\leq \frac{1}{\#(\mathcal{P}_n)}=\frac{1}{l^{n}}$, and in case of $1/l$-uniform measure one has $h_{\mu}(\sigma_{\tau})=\ln l$. Indeed , by Theorem \ref{thm:vp} the variational principle is valid for zip shift maps as continuous functions and, $$h_{top} (\sigma_{\tau} ) = \sup\{h_{\mu}(\sigma_{\tau}) : \mu \in \mathcal{M}(\Sigma,\sigma_{\tau})\}.$$




\section{Square entropy}\label{sec:03}
In \cite{1} the authors give a complete characterization for non-invertible transformations conjugated mod-0 with a one-sided Bernoulli map. However the example \eqref{ex:3.1} and many other examples of this type does not enter in the conjugacy class of a one-sided Bernoulli. Characterization of some class of such uniform n-to-1 examples is one of our main purposes in this work. 	

\begin{definition}[\textbf{Good Image Partition (GIP)}]
	Let $(f,\mu)$ be a measure dynamical system. We say that a finite measurable partition $\mathcal{Q}=\{Q_1,Q_2,\dots,Q_k\}$ is a \textit{Good Image Partition} if the elements of $\mathcal{Q}$ are forward $\mu$-invariant.
\end{definition}
As mentioned in \cite{8}, there exists some correspondence between the measure partitions and the $\sigma$-algebras. 
For a measure dynamics $(f,\mu)$, we say that $\sigma$-algebra $\tilde{\mathcal{B}}$ associated with a GIP, is a \textbf{Good $\sigma$-algebra} if any $A\in \tilde{\mathcal{B}}$ is forward and backward measure invariant. 
\begin{remark}
	In general we may not be able to guarantee the existence of a GIP for any continuous map, but in  Proposition \ref{prop:4}, we shows that for transformations of $(m,l)$-Bernoulli type, the good image partition or GIP exists.
\end{remark}
Let $(f,\mu)$ be a measure dynamical system with a GIP $\mathcal{Q}$. Then set,
\begin{equation}
\mathcal{Q}^{+n}=\bigvee_{i=0}^{n-1} f^{i}(\mathcal{Q}),\,\,\,\,\,\,\,\,\,\, n \geq 1,
\end{equation}
and let $$h_{\mu}^{+}(f,\mathcal{Q})=\lim_{n\to\infty}\frac{1}{n}H_{\mu}(\mathcal{Q}^{+n}),$$
be the entropy of $f$ with respect to good image partition $\mathcal{Q}.$ The $(KS)^-$ entropy of the measure dynamical system $(f,\mu)$ is defined as, 
\begin{equation}\label{eq:h-}
h_{\mu}^+(f)=\sup_{\mathcal{Q}}h_{\mu}^+(f,\mathcal{Q}),
\end{equation}
where the suprimum is taken over all good image partitions. We call this entropy the \textbf{forward} or\textbf{ $(KS)^+$ entropy} of $f$. 
Then "\textit{Square measure entropy}" (or \textbf{S-entropy}) of $(X,f,\mu)$ is defined as follows.
$$h_{S,\mu}(f)=\sqrt{h_{\mu}^+(f) h_{\mu}^-(f)}.$$
Here $h_{\mu}^+(f)$ and $h_{\mu}^-(f)$ are respectively the $(KS)^+$ and $(KS)^-$ entropies defined respectively in \eqref{B-ent} and \eqref{eq:h-}.

The $h_{\mu}^+(f)$ is defined on GIP's which are forward-measure preserving. Indeed, following Theorem is a simple adaptation of Theorem \ref{Thm:S-K}.
\begin{theorem}\label{Thm:SK}
	Let $f$ be a continuous map and  $\tilde{\mathcal{B}}\subseteq\mathcal{B}$ a good $\sigma$-algebra. If there exists $\mathcal{Q}_1<\mathcal{Q}_2<\cdots < \mathcal{Q}_n<\cdots$ a non-decreasing sequence of good image partitions with finite entropy such that $\bigcup_{n=1}^{\infty}\mathcal{Q}_n$ generates
	the good $\sigma$-algebra $\tilde{\mathcal{B}}$, up to measure zero. Then,
	$$h_{\mu}^+(f)=\lim_{n\to +\infty}h_{\mu}^+(f,\mathcal{P}_n).$$
\end{theorem}

\begin{proposition}
	The following proprieties hold for square measure entropy of n-to-1 local homeomorphisms with a good generating partition:
	\begin{enumerate}
		\item $h_{S,\mu}(f^k) = k\,h_{S,\mu}(f) \quad\quad\quad\forall k\in \mathbb{N};$
		\item Let $(f_{1},\mathcal{B}_1, \nu_1)$ and $(f_{2},\mathcal{B}_2,\nu_2)$ be correspondingly, two n-to-1 maps with good generating partitions $\mathcal{Q}_1,\mathcal{Q}_2$. Then $$h_{S,\nu}(f_{1}\times f_{2})\geq h_{S,\nu_S}(f_{1})+h_{S,\nu_2}(f_{2}),$$ with $\nu=\nu_1\times \nu_2.$ The equality happens when both $f_i$ for $i=1,2$ are invertible.
	\end{enumerate}  
	
\end{proposition}
\begin{proof}
	(1) As $h_{\mu}^{+}((f)^k)=k h_{\mu}^{+}(f)$ and as $f$ has a good generating partition, $h_{\mu}^{-}((f)^k)=k h_{\mu}^{-}(f)$. Indeed  $h_{S,\mu}(f^k) = \sqrt{k^ 2\,h_{\mu}^{+}(f)h_{\mu}^{-}(f)}=k\,h_{S,\mu}(f)$.\\
	(2) As we know from KS-entropy :$h_{\nu}^{+}(f_1\times f_2)=h_{\nu_1}^+(f_1)+h_{\nu_2}^+(f_2)$ and similarly one can show that $h_{\nu}^{-}(f_1\times f_2)=h_{\nu_1}^-(f_1)+h_{\nu_2}^-(f_2)$. Indeed,
	\begin{align*}
	h_{S,\nu}(f_{1}\times f_{2})
	&=\sqrt{h_{\nu}^+(f_{1}\times f_{2})h_{\nu}^-(f_{1}\times f_{2})}\\
	&=\sqrt{(h_{\nu_1}^+(f_1)+h_{\nu_2}^+(f_2))(h_{\nu_1}^-(f_1)+h_{\nu_2}^-(f_2))}\\
	&\geq h_{S,\nu_1}(f_1)+ h_{S,\nu_2}(f_2).
	\end{align*}
\end{proof}

Let $f:X\to X$ be an n-to-1 local hommeomorphism and $X$ be a compact connected metric space. 
The zip shift related study of  non-invertible maps with respect to their natural extension,  shows that it is possible to study the topological properties of local hommeomorphisms independent of the choice of the branches. In this work, we use this fact to define the square topological entropy. Let $X^f$ represent the inverse limit space of $X$ \cite{14}.
Let $\tilde{x}\in X^f$ be some fixed forward-backward orbit of point $x\in X$. The \textit{$T^+$-entropy} of $f=f_{\tilde{x}}$ relative to a cover $\alpha$ is defined as, $$h_{top}^+(f,\alpha)=\limsup_{n\rightarrow +\infty}\frac{1}{n}H(\bigvee_{i=0}^{n-1}f_{\tilde{x}}^{i}(\alpha)),$$
and 
$$h_{top}^+(f)=\sup\{h_{top}^+(f_{\tilde{x}},\alpha): \,\alpha\,\, \text{is a finite cover for}\, X\}.$$

The \textit{$T^{-}$-entropy} of $f$ relative to cover $\alpha$ is defined as
$$h_{top}^-(f,\alpha)=\limsup_{n\rightarrow +\infty}\frac{1}{n}H(\bigvee_{i=0}^{n-1}f^{-i}(\alpha)),$$
and 
$$h_{top}^-(f)=\sup\{h_{top}^-(f,\alpha): \,\alpha\,\, \text{is a finite cover for}\, X\}.$$
As $f_{\tilde{x}}$ is a hommeomorphism, one can show with a mild adaptation of the proof, that the following Proposition is valid.
\begin{proposition}\label{Prop:0-1}	\cite{8}
	If $f:X\to X$ is a homeomorphism and $\alpha$ is a generator for $f$, then $h_{top}^\pm(f)=h_{top}^\pm(f,\alpha).$
\end{proposition}
The \textit{Square Topological Entropy} of $f$ is defined as follows.
\begin{equation}\label{Def:STE}
h_{S,top}(f):=\sqrt{h_{top}^+(f) h_{top}^-(f)}.
\end{equation}
Note that when $f$ is an invertible map $h_{top}^+(f)=h_{top}^-(f)=h_{top}(f)$ and indeed $h_{S,top}(f)=h_{top}(f).$

\subsection{Square measure entropy of n-to-1 zip shift maps}\label{subsec:3.1}
Let $Z=\{a_1,\cdots,a_m\}$ and $S=\{0,1,\cdots, l-1\},$ with $m\leq l$ and $\sigma_{\tau}:\Sigma\to \Sigma$ be a (full) zip shift map where $\tau:S\to Z$ is some onto map. Let $\sigma^{L}:\Sigma_{S}\to\Sigma_{S}$ be the two-sided \textbf{left} shift map defined on $\Sigma_{S}=\prod_{-\infty}^{+\infty} S$. Assume that  $P_S=(p_0,\dots,p_{l-1})$ be a probability distribution on $S$  and $(\Sigma_S, \mathcal{B}_S, \mu_S)$ be the probability measure space on $\Sigma_S$ where $\mu_S$ is the invariant Bernoulli measure\cite{012} on $\mathcal{B}_S$ under $\sigma^{L}$. Furthermore, let $\sigma^{R}:\Sigma_{Z}\to\Sigma_{Z}$ be the two-sided \textbf{right} shift map defined on $\Sigma_{Z}=\prod_{-\infty}^{+\infty} Z$. Assume that $P_Z=(p'_{a_1},\dots,p'_{a_m})$ is the probability distribution induced on $Z$ obtained from Proposition \ref{prop:3.1}. Then $(\Sigma_Z, \mathcal{B}_Z, \mu_Z)$ is the probability measure space on $\Sigma_Z$ where $\mu_Z$ is an invariant Bernoulli measure on $\mathcal{B}_Z$ under $\sigma^{R}$.

Using Theorems \ref{Thm:S-K} and \ref{Thm:SK} respectively, calculate the extended Kolmogrov-Sinai entropies $KS^{+}$ and $(KS)^-$ of a uniform n-to-1 zip shift map with $S$ ($\#(S)=l$) and $Z$ ($\#(Z)=m$) alphabets. These entropies become respectively equal with $h_{\nu}^+(\sigma_{\tau})\leq\log l$ and $h_{\nu}^-(\sigma_{\tau})\leq\log m$. The equality holds for the uniform measure case. Note that for such n-to-1 zip shift map:
\begin{enumerate}
	\item $\bigcup_{n=1}^{\infty}\mathcal{C}_{n}$ where $\mathcal{C}_{n}=\bigvee_{i=0}^{n-1} \sigma_{\tau}^{-i}(\mathcal{C}_0)$ and $\mathcal{C}_0$ is the set of all basic cylinder sets of the form $C_{0}^{s_i}, s_i\in S$ generates the $\sigma$-algebra.
	\item $\bigcup_{n=1}^{\infty}\mathcal{C}_{-n}$ where $\mathcal{C}_{-n}=\bigvee_{i=0}^{n-1} \sigma_{\tau}^{i}(\mathcal{C}_{-1})$ and $\mathcal{C}_{-1}$ is the set of all basic cylinder sets of the form $C_{-1}^{s_i}, s_i\in Z$ generates the good Borel  $\sigma$-algebra. 
\end{enumerate}

\begin{lemma}\label{lem:zse}
	Let $(\Sigma, \sigma_{\tau},\mu)$ represent a uniform n-to-1 zip shift measure space with $\sigma_{\tau}:\Sigma\to \Sigma$ and alphabet sets $S,Z$. Then $h_{\mu}^+(\sigma_{\tau})=h_{\mu_S}(\sigma^L)$ and $h_{\mu}^-(\sigma_{\tau})=h_{\mu_S}(\sigma^R)$.
\end{lemma}
\begin{proof}
	Considering above items (1) and (2) and using Theorems \ref{3.2} and \ref{Thm:SK}, one can find that the $(KS)^+$ entropy of a uniform n-to-1 zip shift map on $S,Z$ alphabets, is equal $\log l$ where $l=\#(S)$ and the $(KS)^-$ entropy of a uniform n-to-1 zip shift map on $S,Z$ alphabets, is equal $\log m$ where $m=\#(Z)$. 
\end{proof}

Using Lemma \ref{lem:zse}, we redefine the "\textit{S-entropy}" of an n-to-1 zip shift map $(\Sigma,\sigma_{\tau},\mu)$ as follows. 
\begin{equation}\label{Eq:t-meas}
h_{S,\mu}(\sigma_{\tau}):=\sqrt{h_{\mu_S}(\sigma^{L}) h_{\mu_Z}(\sigma^{R})}.
\end{equation}
Here $h_{\mu_S}(\sigma^{L})$ and $h_{\mu_Z}(\sigma^{R})$ represent respectively the Kolmogrove-Sinai entropies of the two-sided left (i.e. $\sigma^{L}$ defined on $\Sigma_{S}$) and right (i.e. $\sigma^{R}$ defined on $\Sigma_Z$) shift maps. Note that if $\mathcal{B}$ represent the Borel $\sigma$-algebra defined on $\Sigma$, then $\mathcal{B}\subset \mathcal{B}_S\times \mathcal{B}_Z$ and the measure of a basic cylinder set $A\in \mathcal{B}$ considering \eqref{Eq:1} is defined as $\mu(A):=\mu_S(A)\mu_Z(A)$.
\begin{remark}\label{rem:1}
	It is noteworthy that by Lemma \ref{Lem:m-e}, the maximum value of such square entropy is attained by a $\frac{1}{l}$-uniform measure induced by a uniform probability distribution on elements of $S$ (with $\#(S)=l$). Let us denote the maximal entropy measure for $\sigma^L$ by $\mu_S^*$ and the maximal entropy measure for $\sigma^R$ by $\mu_Z^*$. Note that using Lemma \ref{Lem:m-e}, for uniform n-to-1 zip shift maps,
	\begin{equation}\label{Eq:3}
	h_{\mu_Z}\leq h_{\mu_Z^*}\,\text{and}\,h_{\mu_S}\leq h_{\mu_S^*}\Rightarrow h_{\mu_Z}h_{\mu_S}\leq h_{\mu_Z^*}h_{\mu_Z^*}\Rightarrow \sqrt{h_{\mu_Z}h_{\mu_S}}\leq \sqrt{h_{\mu_Z^*}h_{\mu_S^*}}.
	\end{equation}
	
\end{remark}

\begin{example}\label{Ex:4}
	One can calculate the square entropy of transformations (1) and (2) represented in Figure \ref{fig:1} with respect to the uniform probability distribution $1/4$ on elements of $S$. Then the S-entropy for (1) is $h_{S,\mu_S^*}(T)=\sqrt{\ln 2\,\ln 4}.$ However, as for (2) $h_{\mu}^+(T)\neq \ln 2$, their square entropies are not equal.
\end{example}\label{Ex:5}

\subsection{Square topological entropy of n-to-1 zip shift maps}\label{subsec:3.2}
Let $Z$ and $S$ be as before, with $m\leq l$ and $\sigma_{\tau}:\Sigma\to \Sigma$ be an n-to-1 zip shift map with $\tau:S\to Z$. Let $\sigma^{L}:\Sigma_{S}\to\Sigma_{S}$ be the two-sided \textbf{left} shift map defined on $\Sigma_{S}=\prod_{0}^{+\infty} S$ and $\sigma^{R}:\Sigma_{Z}\to\Sigma_{Z}$ be the two-sided \textbf{right} shift map defined on $\Sigma_{Z}=\prod_{-1}^{-\infty} Z$. We define the \textit{Square topological entropy} or in abbreviation the "\textit{S-topological entropy}" of $\sigma_{\tau}$ as follows.
\begin{equation}\label{Eq:t-top}
h_{S,top}(\sigma_{\tau}):=\sqrt[]{h_{top}(\sigma^{L}) h_{top}(\sigma^{R})}.
\end{equation}
As it is known \cite{012}, the topological entropy of the two-sided shift (left or right) on finite alphabets equals to the one-sided shift (left or right). Indeed we have the following lemma.
\begin{lemma}\label{lem:3.9}
	Let $(\Sigma, \sigma_{\tau})$ represent an n-to-1 zip shift map with $\sigma_{\tau}:\Sigma\to \Sigma$ and alphabet sets $S,Z$. Then $h_{top}^+(\sigma_{\tau})=h_{top}(\sigma^L)$ and $h_{top}^-(\sigma_{\tau})=h_{top}(\sigma^R)$.
\end{lemma}


\section{Variational principal and Intrinsic ergodicity}\label{sec:4}
We aim to prove a variational principal for n-to-1 zip shift maps. 
For a zip shift measure space $(\Sigma, \sigma_{\tau}, \mu)$, with $\mu$ being a $ \sigma_{\tau}-$invariant measure. Consider the corresponded two-sided shift measure spaces $(\Sigma_{i},\sigma^{i},\mu_i)$ $\sigma^i:\Sigma_{S}\to\Sigma_{S}$ for $i=L,R$ where $\mu_S$ is induced by $P_S=(p_0,\dots,p_{l-1})$ and $\mu_Z$ is induced by $P_{Z}=(p'_{a_1},\dots,p'_{a_m})$. Recall that by Proposition \ref{prop:3.1} measure $\mu$ is an invariant measure iff $\mu_S$ and $\mu_Z$ satisfy \eqref{Eq:inv-m}.  In what follows by $\mathcal{M}(\Sigma, f)$ we mean the set of all $f-$invariant probability measures.

\begin{theorem}\label{Thm:VP}
	Let $Z$ (with cardinality $m$) and $S$ (with cardinality $l$) be two set of alphabets with $m\leq l$ and $\sigma_{\tau}:\Sigma\to \Sigma$ be a zip shift map corresponded to some  $\tau:S\to Z$. Then,
	$$h_{S,top}(\sigma_{\tau}) \geq \sup\{h_{S,\mu}(\sigma_{\tau}) : \mu\in \mathcal{M}(\Sigma, \sigma_{\tau})\}.$$	
\end{theorem}

\begin{proof}
	Let $\sigma^{L}:\Sigma_{S}\to\Sigma_{S}$  and $\sigma^{R}:\Sigma_{Z}\to\Sigma_{Z}$ respectively represent the  two-sided left and two-sided right shift maps. Then by Corollary \ref{cor:1} the following variational principle is true.
	$$h_{top} (\sigma^j) = \sup\{h_{\mu_i}(\sigma^j) : \mu_i \in \mathcal{M}(\Sigma_i,\sigma^j)\}\,\text{for}\, (j=R,i=Z) \,\underline{or}\, (j=L,\,i=S).$$ 
	Let denote by $\mathcal{M}_*(\Sigma_Z,\sigma^R)\subseteq \mathcal{M}(\Sigma_Z,\sigma^R)$ the subset of $\sigma^R-$ invariant measures which satisfy \eqref{Eq:inv-m}.
	By \eqref{Eq:3} and  \eqref{Eq:t-top} we have,
	\begin{align*}
	h_{S,top}(\sigma_{\tau}) &= \sqrt{h_{top}(\sigma^{L}) h_{top}(\sigma^{R})}\\
	&=\sqrt{\sup_{\mu_S}\{h_{\mu_S}(\sigma^L) : \mu_S \in \mathcal{M}(\Sigma_S,\sigma^L)\}\,\sup_{\mu_Z}\{h_{\mu_Z}(\sigma^R) : \mu_Z \in \mathcal{M}(\Sigma_Z,\sigma^R)\}}\\
	&\geq\sqrt{\sup_{\mu_S}\{h_{\mu_S}(\sigma^L) : \mu_S \in \mathcal{M}(\Sigma_S,\sigma^L)\}\,\sup_{\mu_Z}\{h_{\mu_Z}(\sigma^R) : \mu_Z \in \mathcal{M}_*(\Sigma_Z,\sigma^R)\}}\\
	&=\sup_{\mu}\{h_{S,\mu}(\sigma_{\tau}) : \mu\in \mathcal{M}(\Sigma, \sigma_{\tau})\},
	\end{align*}
	where,
	\begin{equation}\label{Eq:01}
	\mu(C_i^{s_i}):=\left\lbrace\begin{array}{ll}
	\mu_{S}(C_i^{s_i}) & \quad\quad\text{if } s_i\in S; \\
	\mu_{Z}(C_i^{s_i}) & \quad\quad\text{if } s_i\in Z.
	\end{array}\right. 
	\end{equation}
	
\end{proof}

\begin{theorem}\label{Thm:VP1}
	Let $Z$ (with cardinality $m$) and $S$ (with cardinality $l$) be two set of alphabets with $m\leq l$ and $\sigma_{\tau}:\Sigma\to \Sigma$ be an n-to-1 zip shift map corresponded to some  $\tau:S\to Z$. Then,
	$$h_{S,top}(\sigma_{\tau})= \sup\{h_{S,\mu}(\sigma_{\tau}) : \mu\in M(\Sigma, \sigma_{\tau})\}.$$	
\end{theorem}

\begin{proof}
	Consider the uniform probability spaces $(\Sigma_S, \mathcal{B}_S, \mu^{*}_S)$ on $\Sigma_{S}$ and $(\Sigma_Z, \mathcal{B}_Z, \mu^{*}_Z)$ on $\Sigma_Z$.  Then by Corollary \ref{cor:1} and Lemma \ref{Lem:m-e}, for n-to-1 zip shift maps, 
	\begin{align*}
	h_{S,top}(\sigma_{\tau}) &= \sqrt{h_{top}(\sigma^{L}) h_{top}(\sigma^{R})}\\
	&=\sqrt{\sup_{\mu_S}\{h_{\mu_S}(\sigma^L) : \mu_S \in \mathcal{M}(\Sigma_S,\sigma^L)\}\,\sup_{\mu_Z}\{h_{\mu_Z}(\sigma^R) : \mu_Z \in \mathcal{M}(\Sigma_Z,\sigma^R)\}}\\
	&=\sqrt{h_{\mu_S^*}(\sigma^L)\,h_{\mu_Z^*}(\sigma^R)}\\
	&=\sup_{\mu}\{h_{S,\mu}(\sigma_{\tau}) : \mu\in \mathcal{M}(\Sigma, \sigma_{\tau})\},
	\end{align*}
	where  $\mu$ is as \eqref{Eq:01}.
\end{proof}

The following Theorem from \cite{9}, shows that zip shift maps and in special the n-to-1 zip shift maps are strongly mixing and ergodic.

\begin{theorem}\label{Thm:3}
	The zip shift map $\sigma_{\tau}:\Sigma\to \Sigma$ is strongly mixing and ergodic.
\end{theorem}

\begin{proposition}
	Let $\mathcal{P}$ be a finite partition for a zip shift map $\sigma_{\tau}:\Sigma\to \Sigma$ over a Lebesgue probability space $(\Sigma, \mathcal{B}, \mu)$ such that $\mu(\partial(\mathcal{P}))=0$. Then,
	the function $\nu\rightarrow h_{S,\nu}(\sigma_{\tau},P)$ is upper semi-continuous at $\mu$.
\end{proposition}

\begin{proof}
	Let $\mathcal{M}(\Sigma,\sigma_{\tau})$ denote the space of all invariant probability measures of $\sigma_{\tau}$. Considering the weak* topology, it can be shown that the $\mathcal{M}(\Sigma,\sigma_{\tau})$ is compact. To state the upper semi-continuity, we show that $\limsup_{\nu_n\to \mu} h_{S,\nu_n}(\sigma_{\tau})\leq h_{S,\nu}(\sigma_{\tau})$. As it is known \cite{012} the Kolmogrov-Sinai entropy of continuous maps and in special the bilateral Bernoulli shift is upper-semi continuous. Let $\nu_n\to \mu$ (see Remark \eqref{rem:2}), then
	\begin{align*}
	\limsup_{\nu_n\to \mu} h_{S,\nu_n}(\sigma_{\tau})&=\limsup_{{\nu_n}\to \mu} \sqrt{h_{{(\nu_n)}_{Z}}(\sigma^{R})h_{{(\nu_n)}_{S}}(\sigma^{L})}\\
	&=\sqrt{\limsup_{{(\nu_n)}_Z\to \mu_Z}h_{{(\nu_n)}_{Z}}(\sigma^{R})\limsup_{{(\nu_n)}_S\to \mu_S}h_{{(\nu_n)}_{S}}(\sigma^{L})}\\
	&\leq \sqrt{h_{\mu_Z}(\sigma^{R})h_{\mu_S}(\sigma^{L})}=h_{S,\mu}(\sigma_{\tau}).
	\end{align*}
\end{proof}

\begin{theorem}
	The n-to-1 zip shift maps have intrinsic ergodicity with respect to S-entropy.	
\end{theorem}

\begin{proof}
	By the proof of Theorem \ref{Thm:VP1}, we have 
	$$h_{S,top}(\sigma_{\tau})= \sqrt{h_{\mu_S^*}(\sigma^L)\,h_{\mu_Z^*}(\sigma^R)}=h_{S,\mu^{*}	}(\sigma_{\tau}) ,$$
	which	$\mu_{i}^{*}$ for $i=Z,S$ are the measures of maximal entropy for corresponded $\sigma^j,\,j=R,L$. Now by Lemma \ref{Lem:m-e}, any n-to-1 zip shift map has intrinsic ergodicity.
\end{proof}


\section{Classification of uniform n-to-1 $(m,l)$-Bernoulli transformations}\label{sec:5}
The Bernoulli transformations and non-invertible dynamics which are one-sided Bernoulli transformation are studies extensively \cite{10},\cite{1},\cite{013}. Howbeit, as mentioned in Section \ref{sec:03}, Example \ref{ex:3.1} and the class of all such finite-to-1 maps are not one-sided Bernoulli transformations. The interested reader can find more examples of transformations which are not one-sided Bernoulli in \cite{013}.
In order to improve the classification of finite-to-1 transformations, which are not necessarily one-sided Bernoulli, we use the following property. It extends the definition of the two-sided Bernoulli property.
\begin{definition}[\textbf{Extended Bernoulli maps}]
	The measure preserving transformation $f:X\rightarrow X$ defined on a Lebesgue space $(X,\mathcal{B}, \mu)$ is of $(m,l)$-Bernoulli type if it is conjugated mod-0 with a full zip shift map. Here $m=\#(Z)$ and $l=\#(S)$ represent the cardinalities of the symbolic sets $Z$ and $S$ of the associated zip shift map.
\end{definition}
In what follows by an n-to-1 Bernoulli map we mean an n-to-1 measure preserving transformation  of $(m,l)$-Bernoulli type.

Next, we show that for n-to-1 transformations  of $(m,l)$-Bernoulli type, the good image partition or GIP exists. 
\begin{proposition}\label{prop:4}
	Let $(X,f,\mu)$ represent  a transformation  of $(m,l)$-Bernoulli type. Then there exists some partition $Q$ for $X$ such that $Q$ is a good image partition.
\end{proposition}
\begin{proof}
	Let $(X,f,\nu)$ be a transformation  of $(m,l)$-Bernoulli type, which is isomorphic (mod-0) with some zip shift measure dynamics $(\sigma_{\tau},\mu)$ on $m$ ($\#(Z)$) and $l$ ($\#(S)$) alphabets. Denote this isomorphism by $\phi:X\to \Sigma$. Recall that by Definition \ref{I-T}, the map $\phi$ is invertible and bi-measure preserving.  In order to obtain a good image partition on $X$, we recall the basic cylinder sets on the (full) zip shift space $(\Sigma,\sigma_{\tau})$. Let $Z=\{a_1,\dots,a_k\}$ and  $S=\{0,1,\dots,l-1\}$
	and $\mu$ be the invariant measure induced from a probability distribution $P_S=(p_{0},\dots,p_{l-1})$. Then the measurable partition $\mathcal{C}_Z=\{C_{-1}^{a_1},\dots,C_{-1}^{a_k}\}$ is a good image partition.
	Note that conjugacy preserves the degree of maps (i.e. the cardinality of the pre-image of a point). By Proposition \ref{prop:3.1} the basic cylinder set $C_{-1}^{a_i},a_i\in Z$ has measure $\mu(C_{-1}^{a_i})=\sum_{j\in \tau^{-1}(a_i)}p_j$, where $j\in S$. Indeed for any $a_i\in Z$,
	$$\mu(\sigma_{\tau}(C_{-1}^{a_i}))=\mu(C^{a_i}_{-2})=\sum_{j\in \tau^{-1}(a_i)}p_j.$$ 
	Which means that elements of $\mathcal{C}_Z=\{C_{-1}^{a_i}:a_i\in Z\}$ are forward invariant and indeed $\mathcal{C}_Z$ is a GIP. Now one uses $\phi$ to pull back the measurable partition $\mathcal{C}_Z$ into a measurable partition $Q$ on $X$. As $\phi$ is a bi-measure preserving map, the elements $Q_i=\phi^{-1}(C_{-1}^{a_i})$ of $Q$ becomes forward $\nu-$measure invariant and $Q$ is a GIP.
\end{proof}

\begin{theorem}
	Any Bernoulli transformation is  of $(m,l)$-Bernoulli type.
\end{theorem}
\begin{proof}
	It is not difficult to verify that by Definition \eqref{ZS} any shift homeomorphism defined on a shift space with $l$ alphabets is a zip shift map with $S=Z$ and with $\tau(x)=id(x)$. Indeed any Bernoulli transformation is conjugated with a 1-to-1 zip shift map and is  of $(m,l)$-Bernoulli type with $m=l$.
\end{proof}
Observe that the inverse of above theorem is not valid in general.

\begin{proposition}\label{prop:2}
	The square entropy is preserved under measure theoretical conjugacy in the class of transformations of $(m,l)$-Bernoulli type.
\end{proposition}

\begin{proof}
	The proof of this theorem arises from the fact that $(KS)^-$ and $(KS)^+$ entropies are preserved by conjugacy (mod-0).
\end{proof}
%
As it is known, the Ornstein Theorem \cite{8},\cite{10} is not necessarily valid for non-invertible case. In \cite{1} the authors give a complete characterization of n-to-1 maps with $\#(Z)=1$ (i.e. one-sided Bernoulli transformations). The following Theorem can be seen as an application of the square entropy for the case of $n>1, \#(Z)>1$.
\begin{theorem}\label{thm:ztc}
	The uniform n-to-1 transformations with $\#(Z)>1$ are isomorphic if and only if they have the same square entropy.
\end{theorem}

\begin{proof}
	Using Proposition \ref{prop:2}, the square entropy is preserved by conjugacy (mod-0). Indeed it is enough to show that two transformations  of $(m,l)$-Bernoulli type,  with the same square entropy are isomorphic.\\
	$\Leftarrow)$
	Let $(f_1,\mu_1)$ and $(f_2,\mu_2)$ be uniform n-to-1 maps  of $(m,l)$-Bernoulli type. There exist zip shift maps  $(\sigma_1,\nu_1)$ with $(m_1,l_1)$ symbols and $(\sigma_2,\nu_2)$ with $(m_2,l_2)$ symbols associated respectively to $f_1$ and $f_2$. These maps have the square entropies, as follows. \begin{align*}
	h_{T,\mu_1}(f_1)=h_{T,\nu_1}(\sigma_1)=\sqrt{\log m_1\log l_1}.\\ h_{T,\mu_2}(f_2)=h_{T,\nu_2}(\sigma_2)=\sqrt{\log m_2 \log l_2}.
	\end{align*} 
	Notice that once we have a n-to-1 zip shift map, $\#(S)=n(\#(Z))$.
	Indeed the map $\ln m\,\ln l$ with $m,l>1$ is injective. In other words, these entropies could be equal if and only if $m_1=m_2$. This shows that uniform n-to-1 maps  of $(m,l)$-Bernoulli type and equal square entropy are necessarily isomorphic.
\end{proof}


\subsection*{Conflict of Interest:} The authors declare that there is no conflict of interest.

\subsection*{Data Availability}
	Data sharing is not applicable to this article as no data sets were generated or analyzed during the current study.
	
\subsection*{Funding: }This research was partially funded by FAPEMIG-Brazil.


\end{document}